\documentclass[reqno, 11pt]{amsart}
\usepackage{amsmath,mathrsfs}
\usepackage{amsfonts}
\usepackage{amssymb}
\usepackage{graphicx,color}
\usepackage[square,numbers]{natbib}
\usepackage{verbatim}
\newtheorem{theorem}{Theorem}
\newtheorem{claim}[theorem]{Claim}
\newtheorem{corollary}[theorem]{Corollary}
\newtheorem{lemma}[theorem]{Lemma}
\newtheorem{proposition}[theorem]{Proposition}

\theoremstyle{definition}
\newtheorem{definition}[theorem]{Definition}

\theoremstyle{remark}

\newtheorem{example}[theorem]{Example}

\numberwithin{theorem}{section}
\numberwithin{equation}{section}

\def\XXint#1#2#3{{\setbox0=\hbox{$#1{#2#3}{\int}$}
     \vcenter{\hbox{$#2#3$}}\kern-.5\wd0}}

\newcommand{\dd}{\; \mathrm{d}}

\newcommand{\bbG}{\mathbb{G}}
\newcommand{\bbR}{\mathbb{R}}
\newcommand{\bbN}{\mathbb{N}}

\begin{document}
\title[Porosity, Differentiability and Pansu's Theorem]{Porosity, Differentiability and Pansu's Theorem}
\author[Andrea Pinamonti]{Andrea Pinamonti}
\address[Andrea Pinamonti]{Universit\'a degli Studi di Trento, Dipartimento di Matematica, via Sommarive 14, 38123 Povo (TN), Italy}
\email{Andrea.Pinamonti@gmail.com}
\author[Gareth Speight]{Gareth Speight}
\address[Gareth Speight]{University of Cincinnati, Department of Mathematical Sciences, 2815 Commons Way, Cincinnati 45221, United States}
\email{Gareth.Speight@uc.edu}

\date{\today}

\renewcommand{\subjclassname} {\textup{2010} Mathematics Subject Classification}
\subjclass[]{28A75, 43A80, 49Q15, 53C17}


\keywords{Pansu's theorem, Porous set, Carnot group}

\begin{abstract}
We use porosity to study differentiability of Lipschitz maps on Carnot groups. Our first result states that directional derivatives of a Lipschitz function act linearly outside a $\sigma$-porous set. The second result states that irregular points of a Lipschitz function form a $\sigma$-porous set. We use these observations to give a new proof of Pansu's theorem for Lipschitz maps from a general Carnot group to a Euclidean space.
\end{abstract}

\maketitle

\section{Introduction}

A Carnot group (Definition \ref{Carnotdef}) is a connected and simply connected Lie group whose Lie algebra admits a stratification. Carnot groups have translations, dilations, Haar measure and points are connected by horizontal curves (Definition \ref{horizontalcurve}), which are used to define the Carnot-Carath\'eodory distance. With so much structure, the study of analysis and geometry in Carnot groups is an active and interesting research area \cite{ABB, B94, BLU, CDPT07, Gro96, Mon02, SC16, Vit14}. 

The geometry of Carnot groups is highly non-trivial. For instance, any Carnot group (except for Euclidean spaces themselves) contains no subset of positive measure that is bi-Lipschitz equivalent to a subset of a Euclidean space \cite{Sem96}. This follows from Pansu's theorem (Theorem \ref{PansuThm}), a generalization of Rademacher's theorem, which asserts that Lipschitz maps between Carnot groups are differentiable almost everywhere \cite{PP, Mag01}. Differentiability of maps between Carnot groups is defined like that of maps between Euclidean spaces, but Euclidean translations, dilations and distances are replaced by their analogues in Carnot groups. 

Many interesting geometric and analytic problems have been studied in the context of Carnot groups. For example, a geometric notion of intrinsic Lipschitz function between subgroups of a general Carnot group was introduced in \cite{FSC2} to study rectifiable sets \cite{FSSC2,Mag} and minimal surfaces \cite{CMPSC1,CMPSC2,SCV}. Moreover, Carnot groups have been applied to study degenerate equations, control theory and potential theory \cite{BLU}. More recently, they were also considered in applied mathematics, such as mathematical finance, theoretical computer science and  mathematical models for neurosciences \cite{CSP}.


In this paper, we generalize to Carnot groups two statements asserting that a set of exceptional points concerning directional derivatives is $\sigma$-porous, and use them to give a new proof of Pansu's theorem for Lipschitz maps from a general Carnot group to a Euclidean space.

A set in a metric space is (upper) porous (Definition \ref{def_porous}) if each of its points sees nearby relatively large holes in the set on arbitrarily small scales. A set is $\sigma$-porous if it is a countable union of porous sets. Porous sets have applications to topics like cluster set theory and differentiability of Lipschitz or convex mappings. Note there are several other types of porosity with their own applications. Properties and applications of porous sets are thoroughly discussed in the survey articles \cite{Zaj87, Zaj05}. 

Stating that an exceptional set is $\sigma$-porous is useful because $\sigma$-porous sets in metric spaces are of first category and have measure zero with respect to any doubling measure. The second statement follows, as in the Euclidean case, using the doubling property and the fact that the Lebesgue density theorem holds for doubling measures. Proving that a set is $\sigma$-porous usually gives a stronger result than showing it is of first category or has measure zero: any topologically complete metric space without isolated points contains a closed nowhere dense set which is not $\sigma$-porous, and $\mathbb{R}^{n}$ contains a closed nowhere dense set of Lebesgue measure zero which is not $\sigma$-porous \cite{Zaj87}.

Porosity has been used to study differentiability of Lipschitz mappings, even very recently. Indeed, \cite{LP03} gives a version of Rademacher's theorem for Frech\'{e}t differentiability of Lipschitz mappings on Banach spaces in which porous sets are $\Gamma$-null. Roughly, a set is $\Gamma$-null if it meets typical infinite dimensional $C^1$ surfaces in measure zero. Applications of porosity to study differentiability in infinite dimensional Banach spaces are thoroughly discussed in the recent book \cite{LPT12}. In the finite dimensional setting, porosity and construction of large directional derivatives were recently used in \cite{PS15} to obtain the following result: for any $n>1$, there exists a Lebesgue null set in $\mathbb{R}^{n}$ containing a point of differentiability for every Lipschitz mapping from $\mathbb{R}^{n}$ to $\mathbb{R}^{n-1}$. Directional derivatives also played a key role in the proof of the following result in a Carnot group: there exists a measure zero set $N$ in the Heisenberg group $\mathbb{H}^{n}$ such that every Lipschitz map $f\colon \mathbb{H}^{n}\to \mathbb{R}$ is Pansu differentiable at a point of $N$ \cite{PinS15}.

We now briefly describe the results of this paper. 

For a Lipschitz function $f\colon \mathbb{R}^{n}\to \mathbb{R}$, directional derivatives need not act linearly: $f'(x,u+v)$ may not agree with $f'(x,u)+f'(x,v)$ (see Example \ref{nonlineardd} for a simple example). However, \cite[Theorem 2]{PZ01} asserts that such an implication does hold at points $x$ outside a $\sigma$ porous (even $\sigma$-directionally porous) set depending on $f$, even for Lipschitz maps on separable Banach spaces. We prove a similar statement in Carnot groups: directional derivatives (in horizontal directions) for Lipschitz maps act linearly outside a $\sigma$-porous set (Theorem \ref{porosityderivatives}). Note that in Carnot groups it only makes sense to consider directional derivatives in horizontal derivatives, since the composition of a Lipschitz function with a non-horizontal curve may not be Lipschitz, and consequently may fail to be differentiable. To relate directional derivatives in the different directions and hence prove Theorem \ref{porosityderivatives}, we prove and apply Lemma \ref{ballboxtype} which shows how to move in a direction $U+V$ by repeatedly moving in horizontal directions $U$ and $V$. Such a statement is non-trivial, since $U$ and $V$ are left-invariant vector fields whose direction depends upon the point where they are evaluated. Roughly, the porous sets appearing in Theorem \ref{porosityderivatives} are sets of points at which (on some scale and to some specified accuracy) a Lipschitz function is approximated by directional derivatives in directions $U$ and $V$, but is not well approximated by the sum of the directional derivatives in the direction $U+V$. Our proof of Theorem \ref{porosityderivatives} is based on the proof of \cite[Theorem 2]{PZ01}. The main differences are an application of Lemma \ref{distances} to control the Carnot-Carath\'eodory distance and Lemma \ref{ballboxtype} to account for the fact that in a Carnot group moving in direction $U$ then $V$ is not the same as moving in direction $U+V$. Further, the conclusion of \cite[Theorem 2]{PZ01} was (Gateaux) differentiability which, as discussed below, no longer holds in our setting.

If a Lipschitz map $f\colon \mathbb{R}^{n}\to \mathbb{R}$ has all directional derivatives at a point $x\in \mathbb{R}^n$ and directional derivatives at $x$ act linearly, then by definition $f$ is differentiable at $x$. In Carnot groups we observe that existence and linear action of only directional derivatives in horizontal directions does not imply Pansu differentiability (Example \ref{linearnotdiff}). To pass from directional derivatives to Pansu differentiability, we first give another application of porosity. 

A point $x$ is a regular point of a Lipschitz function $f$ if whenever a directional derivative of $f$ at $x$ in some direction exists, then that same directional derivative also controls changes in $f$ along parallel lines with the same direction close to $x$ (Definition \ref{defregularpoint}). Proposition \ref{irregularporous} adapts \cite[Proposition 3.3]{LP03} to Carnot groups, showing that the set of points at which a Lipschitz function is not regular is $\sigma$-porous. 

We next prove Theorem \ref{differentiableatlast}. This states that if $f\colon \mathbb{G}\to \mathbb{R}$ is Lipschitz and $x$ is a regular point at which directional derivatives exist and act linearly, then $f$ is Pansu differentiable at $x$. To prove this we use Lemma \ref{pathlemma}, which provides a relatively short horizontal path from $x$ to any nearby point which is a concatenation of lines in directions of a basis of $V_{1}$. We use regularity of the point $x$ to estimate changes along these lines using the directional derivatives at $x$, then linear action of directional derivatives to show that the derivative is indeed a group linear map. While our definition of regular point is a generalization of the corresponding notion in Banach spaces, this application is original.

As a consequence of Theorem \ref{porosityderivatives}, Proposition \ref{irregularporous} and Theorem \ref{differentiableatlast}, we obtain Corollary \ref{directionaltodiff}. This asserts that existence of directional derivatives implies Pansu differentiability outside a $\sigma$-porous set. Since it is not hard to show that directional derivatives of a Lipschitz function exist almost everywhere (Lemma \ref{directionalae}) and porous sets have measure zero, we obtain a new proof of Pansu's Theorem (Corollary \ref{Pansucor}) for mappings from Carnot groups to Euclidean spaces.

One can ask if the results of this paper generalize to mappings between arbitrary Carnot groups. The authors chose to investigate the case of Euclidean targets because it resembles more closely the Banach space case yet already calls for interesting new techniques. Our proofs do not immediately generalize to Carnot group targets. 
The authors intend to investigate the more general case in future works.

\medskip

\noindent \textbf{Acknowledgement.} 
The authors thank the referee for his/her useful comments and suggestions. The authors also thank Enrico Le Donne and Valentino Magnani for interesting discussions and suggestions.

\section{Preliminaries}

We now briefly describe the main notions and results used in the paper; the reader can consult \cite{BLU} for more details. 

\subsection{Carnot groups}

Recall that a Lie group is a smooth manifold which is also a group for which multiplication and inversion are smooth. The Lie algebra associated to a Lie group is the space of left invariant vector fields equipped with the Lie bracket $[\cdot, \cdot]$ defined on smooth functions by
\[[X,Y](f)=X(Y(f))-Y(X(f)).\]
We also denote the direct sum of vector spaces $V$ and $W$ by $V\oplus W$.

\begin{definition}\label{Carnotdef}
A simply connected finite dimensional Lie group $\bbG$ is said to be a \emph{Carnot group of step $s$} if its Lie algebra $\mathfrak{g}$ is \emph{stratified of step $s$}, this means that there exist linear subspaces $V_1,...,V_s$ of $\mathfrak{g}$ such that
\[\mathfrak{g}=V_1\oplus \cdots \oplus V_s\]
with
\[[V_1,V_{i}]=V_{i+1} \mbox{ if }1\leq i\leq s-1, \mbox{ and } [V_1,V_s]=\{0\}.\]
Here $[V_1,V_i]:=\mathrm{span}\{[a,b]: a\in V_1,\ b\in V_i\}.$
\end{definition}

It can be shown that Definition \ref{Carnotdef} implies $[V_{i},V_{j}]\subset V_{i+j}$ if $i+j\leq s$ and $[V_{i},V_{j}]=\{0\}$ if $i+j>s$ \cite[Proposition 1.1.7]{BLU}. 

Let $m_i:=\dim(V_i)$, $h_i:=m_1+\dots +m_i$ for $1\leq i\leq s$, and $h_{0}:=0$. A basis $X_1,\ldots, X_n$ of $\mathfrak{g}$ is \emph{adapted to the stratification} if $X_{h_{i-1}+1},\ldots, X_{h_{i}}$ is a basis of $V_i$ for $1\leq i \leq s$. We define $m:=m_1=\dim(V_{1})$ and note $n=h_s$. The basis $X_{1}, \ldots, X_{m}$ of $V_{1}$ induces an inner product $\omega(\cdot,\cdot)$ on $V_1$ for which $X_{1}, \ldots, X_{m}$ is orthonormal. We denote by $\omega (\cdot)$ the norm on $V_{1}$ induced by this inner product.

We recall the exponential map $\exp \colon \mathfrak{g}\to \mathbb{G}$ is defined by $\exp(X)=\gamma(1)$ where
$\gamma \colon[0,1]\to \mathbb{G}$ is the unique solution to $\gamma'(t)=X(\gamma(t))$ and $\gamma(0)=0$. 
The exponential map is a diffeomorphism between $\mathbb{G}$ and $\mathfrak{g}$. Using the basis $X_{1}, \ldots, X_{n}$ to identify $\mathfrak{g}$ with $\mathbb{R}^n$, we can identify $\mathbb{G}$ with $\mathbb{R}^{n}$ by the correspondence:
\[ \exp(x_{1}X_{1}+\ldots +x_{n}X_{n})\in \mathbb{G} \longleftrightarrow (x_{1}, \ldots, x_{n})\in \mathbb{R}^{n}.\]
Using this identification of $\mathbb{G}$ with $\mathbb{R}^n$, \cite[Corollary 1.3.19]{BLU} states that if $h_{l-1}< j\leq h_l$ for some $1\leq l\leq s$ then:
\begin{align}\label{vfexp}
X_j(x)=\partial_j+\sum_{i>h_l}^n q_{i,j}(x)\partial_i,
\end{align}
where $q_{i,j}$ are suitable homogeneous polynomials completely determined by the group law in $\mathbb{G}$. Using \eqref{vfexp} we can identify $X_j(x)$ with the vector
$e_j+\sum_{i>h_l}^n q_{i,j}(x)e_i$ where $e_j$ is the $j$'th element of the canonical basis of $\mathbb{R}^n$. Denote by $p\colon \mathbb{R}^{n} \to \mathbb{R}^{m}$ the projection onto the first $m$ coordinates, given by $p(x)=(x_{1}, \ldots, x_{m})$. Then $p(X_{j}(x))$ is independent of $x\in \mathbb{G}$, so we can unambiguously define $p(X_j)=p(X_j(x))=e_j$ for every $x\in\mathbb{G}$ and $j=1,\ldots, m$. We extend this definition by linearity to all $V_1$, in particular $p(U(x))=p(U(y))$ for every $x,y\in \mathbb{R}^n$, $U\in V_1$.

To compute the group law in coordinates, it is possible to define a mapping $\diamond\colon \mathfrak{g}\times \mathfrak{g}\to \mathfrak{g}$ for which $(\mathfrak{g}, \diamond)$ is a Lie group and $\exp \colon (\mathfrak{g},\diamond)\to (\mathbb{G},\cdot)$ is a group isomorphism \cite[Theorem 2.2.13]{BLU}, in particular:
\begin{align}\label{usefuldiamond}
\exp(X)\exp(Y)=\exp(X\diamond Y)\quad \mbox{ for all } X,Y\in\mathfrak{g}.
\end{align}
The Baker-Campbell-Hausdorf formula gives a formula for $\diamond$:
\begin{align}\label{BCH}
X\diamond Y= X+Y+\frac{1}{2}[X,Y]+\frac{1}{12}([X,[X,Y]]+[Y,[Y,X]]) + \ldots,
\end{align}
where the higher order terms are nested commutators of $X$ and $Y$.

Denoting points of $\mathbb{G}$ by $(x_{1}, \ldots, x_{n})\in \mathbb{R}^{n}$, the \emph{homogeneity} $d_i\in\bbN$ of the variable $x_i$ is defined by
\[ d_j=i \quad\text {whenever}\; h_{i-1}+1\leq j\leq h_{i}.\]
For any $\lambda >0$, the \emph{dilation} $\delta_\lambda\colon \bbG\to\bbG$, is defined in coordinates by
\[\delta_\lambda(x_1,...,x_n)=(\lambda^{d_1}x_1,...,\lambda^{d_n}x_n)\]
and satisfies $\delta_{\lambda}(xy)=\delta_{\lambda}(x)\delta_{\lambda}(y)$. Using the exponential map, dilations satisfying $\exp \circ \delta_{\lambda} = \delta_{\lambda}\circ \exp$ can be defined on $\mathfrak{g}$. These satisfy $\delta_{\lambda}(E)=\lambda^{i}E$ if $E\in V_{i}$ for some $1\leq i\leq s$.

A \emph{Haar measure} on $\mathbb{G}$ is a non-trivial Borel measure $\mu$ on $\mathbb{G}$ satisfying $\mu(gE)=\mu(E)$ for any $g\in \mathbb{G}$ and Borel set $E\subset \mathbb{G}$. Such a measure is unique up to scaling by a positive constant, so sets of measure zero are defined without ambiguity. Identifying $\bbG$ with $\mathbb{R}^n$, any Haar measure is simply a constant multiple of $n$ dimensional Lebesgue measure $\mathcal{L}^{n}$.

\subsection{Carnot-Carath\'eodory distance}

Recall that a curve $\gamma\colon [a,b]\to \mathbb{R}^{n}$ is \emph{absolutely continuous} if it is differentiable almost everywhere, $\gamma' \in L^{1}[a,b]$ and
\[\gamma(t_{2})=\gamma(t_{1})+\int_{t_{1}}^{t_{2}} \gamma'(t)\dd t\]
whenever $t_{1}, t_{2}\in [a,b]$.

\begin{definition}\label{horizontalcurve}
An absolutely continuous curve $\gamma\colon [a,b]\to \mathbb{G}$ is \emph{horizontal} if there exist $u_{1}, \ldots, u_{m}\in L^{1}[a,b]$ such that
\[\gamma'(t)=\sum_{j=1}^{m}u_{j}(t)X_{j}(\gamma(t))\]
for almost every $t\in [a,b]$. We define the horizontal length of such a curve $\gamma$ by:
\[L(\gamma)=\int_{a}^{b}|u(t)|\dd t,\]
where $u=(u_{1}, \ldots, u_{m})$ and $|\cdot|$ denotes the Euclidean norm on $\mathbb{R}^{m}$.
\end{definition}

The Chow-Rashevskii Theorem asserts that any two points of $\mathbb{G}$ can be connected by horizontal curves \cite[Theorem 9.1.3]{BLU}.

\begin{definition}
The \emph{Carnot-Carath\'eodory distance} $d$ on $\mathbb{G}$ is defined by:
\[d(x,y)=\inf \{ L(\gamma) : \gamma \colon [0,1]\to \mathbb{G} \mbox{ horizontal joining }x\mbox{ to }y \}\]
for $x, y\in \mathbb{G}$.
\end{definition}

It is well-known that the Carnot-Carath\'eodory distance satisfies the relations $d(zx,zy)=d(x,y)$ and $d(\delta_{r}(x),\delta_{r}(y))=rd(x,y)$ for $x, y, z\in \mathbb{G}$ and $r>0$. It induces on $\mathbb{G}$ the same topology as the Euclidean distance but is not bi-Lipschitz equivalent to the Euclidean distance. For convenience, we let $d(x):= d(x,0)$. It follows from the definition of the exponential map that $d(\exp(tE))\leq |t|\omega (E)$ whenever $t\in \mathbb{R}$ and $E\in V_{1}$.

A \emph{homogeneous norm} on $\mathbb{G}$ is a continuous function $D\colon \mathbb{G}\to [0,\infty)$ such that $D(\delta_{\lambda}(x))=\lambda D(x)$ for every $\lambda>0$ and $x\in \mathbb{G}$, and $D(x)>0$ if and only if $x\neq 0$. We will mostly use the homogeneous norm $d(x)=d(x,0)$, but it is also useful to consider the homogeneous norm given by the explicit formula:
\begin{equation}\label{auxnorm}
\|x\|:=\left( \sum_{i=1}^{s} |x^{i}|^{2r!/i} \right)^{\frac{1}{2s!}}
\end{equation}
where $x=(x^1,\ldots, x^s)\in \mathbb{R}^{m}\times \cdots \times\mathbb{R}^{m_s}=\mathbb{R}^n$ and $|x^i|$ denotes the Euclidean norm on $\mathbb{R}^{m_i}$. By \cite[Proposition 5.1.4]{BLU} there exists $c>0$ such that
\begin{align}\label{equivdist}
c^{-1} \|x\|\leq d(x)\leq c\| x\|\quad \mbox{for every }x\in\mathbb{G}.
\end{align}

\subsection{Directional derivatives and Pansu differentiability}

If $x\in \mathbb{G}$ and $E\in V_{1}$ is horizontal then the map $t\mapsto x\exp(tE)$ is Lipschitz. Consequently if $f\colon \mathbb{G}\to \mathbb{R}$ is Lipschitz then the composition $t\mapsto f(x\exp(tE))$ is a Lipschitz mapping from $\mathbb{R}$ to itself, hence differentiable almost everywhere. Thus it makes sense to define directional derivatives of Lipschitz maps $f\colon \mathbb{G}\to \mathbb{R}$ in horizontal directions. Note, however, that if $E\in \mathfrak{g}\setminus V_{1}$ then the composition $t\mapsto f(x\exp(tE))$ may not be Lipschitz. In this paper we only consider directional derivatives in horizontal directions, as in \cite{PinS15}.

\begin{definition}\label{directionalderivative}
Suppose $f\colon \mathbb{G}\to \mathbb{R}$, $x\in \mathbb{G}$ and $E\in V_{1}$. We say that $f$ is \emph{differentiable at $x$ in direction $E$} if the limit
\[Ef(x)=\lim_{t\to 0}\frac{f(x\exp(tE))-f(x)}{t}\]
exists. 
\end{definition}

Suppose $f\colon \bbG\to \bbR$, $x\in \mathbb{G}$ and $X_{j}f(x)$ exists for every $1\leq j\leq m$. Then we define the \emph{horizontal gradient} of $f$ at $x$ by
\begin{equation*}
\nabla_{H}f(x):=\sum_{i=1}^{m}(X_{i}f(x))X_{i}(x),
\end{equation*}
which can be represented in coordinates by $(X_1f(x),...,X_{m}f(x)) \in \mathbb{R}^{m}$.

Let $\widetilde{\mathbb{G}}$ be a Carnot group with distance $\widetilde{d}$ and dilations $\widetilde{\delta}_r$. 

\begin{definition}
A map $L\colon \mathbb{G}\to \widetilde{\mathbb{G}}$ is \emph{group linear} (or is a \emph{Carnot homomorphism}) if $L(xy)=L(x)L(y)$ and $L(\delta_{r}(x))=\widetilde{\delta}_r(L(x))$ whenever $x, y\in \mathbb{G}$ and $r>0$.

A map $f\colon \mathbb{G}\to \widetilde{\mathbb{G}}$ is \emph{Pansu differentiable} at $x\in \mathbb{G}$ if
there exists a group linear map $L\colon \mathbb{G}\to \widetilde{\mathbb{G}}$ such that
\[ \lim_{h\to 0} \frac{\widetilde{d}(f(x)^{-1}f(xh), L(h))}{d(h)}= 0.\]
If such a map $L$ exists then it is unique and we denote it by $df$.
\end{definition}

The following fundamental result generalizes Rademacher's theorem to Carnot groups and is due to Pansu \cite{PP}.

\begin{theorem}[Pansu's Theorem]\label{PansuThm}
Let $f\colon \mathbb{G}\to \widetilde{\mathbb{G}}$ be a Lipschitz map. Then $f$ is Pansu differentiable almost everywhere.
\end{theorem}

We fix a Carnot group $\mathbb{G}$ and an adapted basis $X_1,\ldots, X_n$ of $\mathfrak{g}$, with corresponding inner product norm $\omega$, throughout this paper.

\subsection{Porous sets}

We now define porous sets and $\sigma$-porous sets; for more information see the extensive survey articles \cite{Zaj87, Zaj05}. Intuitively, a set is porous if every point of the set sees relatively large holes in the set on arbitrarily small scales. We denote by $B(x,r)$ the open ball of centre $x$ and radius $r>0$ in a metric space.

\begin{definition}\label{def_porous}
Let $(M, \rho)$ be a metric space, $E\subset M$ and $a\in M$. We say that  $E$ is \emph{porous at $a$} if there exist $\lambda>0$ and a sequence $x_{n}\to a$ such that
\[B(x_n,\lambda \rho(a,x_n))\cap E=\varnothing\]
for every $n\in \mathbb{N}$. 

A set $E$ is \emph{porous} if it is porous at each point $a\in E$ with $\lambda$ independent of $a$. A set is \emph{$\sigma$-porous} if it is a countable union of porous sets. 
\end{definition}

Porous sets in Carnot groups have measure zero. This follows from the fact that Haar measure on Carnot groups is Ahlfors regular, hence doubling, so the Lebesgue differentiation theorem applies \cite[Theorem 1.8]{Hei01}. If a porous set had positive measure then it would have a Lebesgue density point, which is impossible due to the presence of relatively large holes, which have relatively large measure, on arbitrarily small scales. 

\section{Directional derivatives acting linearly}

Even for a Lipschitz function $f\colon \mathbb{R}^{2}\to \mathbb{R}$ with all directional derivatives $f'(x,v)$ at a point $x\in \mathbb{R}^2$, directional derivatives need not act linearly: $f'(x,a_{1}v_{1}+a_{2}v_{2}) \neq a_{1}f'(x,v_{1})+a_{2}f'(x,v_{2})$ in general. Consequently, existence of directional derivatives alone does not suffice for differentiability. We illustrate this with a simple example.

\begin{example}\label{nonlineardd}
Define $f\colon \mathbb{R}^2 \to \mathbb{R}$ by $f(x,y)=\min (x,y)$. Then $f$ is a Lipschitz function with all directional derivatives at $(0,0)$. However, $f$ is not differentiable at $(0,0)$, since:
\[f'((0,0),(1,1))=1\neq 0=\frac{\partial f}{\partial x}(0,0)+\frac{\partial f}{\partial y}(0,0).\]
\end{example}

If $x$ lies outside a $\sigma$-porous set depending on $f$, then directional derivatives do act linearly: if $f'(x,v_{1})$ and $f'(x,v_{2})$ exist then $f'(x,a_{1}v_{1}+a_{2}v_{2})$ exists and is equal to $a_{1}f'(x,v_{1})+a_{2}f'(x,v_{2})$. Such a statement holds even for Lipschitz maps from a separable Banach space $X$ to a Banach space $Y$ \cite[Theorem 2]{PZ01}. In this section we prove an analogue in Carnot groups, showing that directional derivatives act linearly outside a $\sigma$-porous set (Theorem \ref{porosityderivatives}).

\subsection{Geometric lemmas}

We now give several results describing the geometry of Carnot groups, essential for our study of porosity and differentiability. Before proving the first geometric lemma, we state an estimate for the norm of a commutator of group elements \cite[Lemma 2.13]{FS}. Recall that $\mathbb{G}$ is a Carnot group of step $s$, $h_s:=n$ and the homogeneous norm $\|\cdot \|$ was defined in \eqref{auxnorm}. The following lemma is stated using a different homogeneous norm in \cite{FS}. However the desired statement follows from \cite[Lemma 2.13]{FS} because the ratio of any two homogeneous norms is bounded.

\begin{lemma}\label{stimagroup}
There is a constant $C>0$ such that
\begin{align*}
\|x^{-1}yx\|\leq C\Big(\|y\|+ \|x\|^{\frac{1}{s}}\|y\|^{\frac{s-1}{s}}+\|x\|^{\frac{s-1}{s}}\|y\|^{\frac{1}{s}}\Big)\quad \mbox{for } x,y\in\mathbb{G}.
\end{align*}
\end{lemma}

The first geometric lemma bounds the Carnot-Carath\'eodory distance between points obtained by flowing from two nearby points in the direction of the same horizontal vector field.

\begin{lemma}\label{distances}
Suppose $\lambda\in (0,1)$ and $t\in (-1,1)$. Let $x, y\in \mathbb{G}$ satisfy $d(x,y)\leq \lambda |t|$ and $U\in V_{1}$. Then, there exists $C=C_{1}>0$ such that
\[d(x\exp(tU),y\exp(tU))\leq C \lambda^{1/s}|t|\max(1,\omega(U)).\]
\end{lemma}

\begin{proof}
Throughout the proof $C$ will denote a positive constant depending only on $\mathbb{G}$ and possibly different from line to line.
Using equation \eqref{equivdist}, we can estimate with the equivalent homogeneous norm $\|\cdot \|$ instead of $d(\cdot)$. Without loss of generality assume $y=0$ and $s\geq 2$. Indeed, if $s=1$ then $\mathbb{G}\equiv\mathbb{R}^n$ and 
\[\|\exp(tU)^{-1}x\exp(tU)\|=\|x\|\leq C\lambda |t|\]
so the statement follows. Now suppose $s\geq 2$. Using Lemma \ref{stimagroup} and the bound $\|x\|\leq C\lambda |t|$ we have:
\begin{align*}
\|\exp(tU)^{-1}x\exp(tU)\|&\leq C\Big(\|x\|+\|x\|^{\frac{1}{s}}\|\exp(tU)\|^{\frac{s-1}{s}}+\|x\|^{\frac{s-1}{s}}\|\exp(tU)\|^{\frac{1}{s}}\Big)\\
&=C\Big(\|x\|+ \|x\|^{\frac{1}{s}}|t|^{\frac{s-1}{s}}\omega(U)^{\frac{s-1}{s}}+\|x\|^{\frac{s-1}{s}}|t|^{\frac{1}{s}}\omega(U)^{\frac{1}{s}}\Big)\\
&\leq C |t|\Big(\lambda +\lambda^{\frac{1}{s}}\omega(U)^{\frac{s-1}{s}}+\lambda^{\frac{s-1}{s}}\omega(U)^{\frac{1}{s}}\Big)\\
&\leq C |t| \lambda^{\frac{1}{s}}\Big(1+\omega(U)^{\frac{s-1}{s}}+\omega(U)^{\frac{1}{s}}\Big),
\end{align*}
where in the last inequality we used $\lambda\in (0,1)$ and $s\geq 2$. The conclusion easily follows.
\end{proof}

We next recall \cite[Lemma 19.1.4]{BLU}, as summarized in \cite[page 717]{BLU}. Keeping in mind \eqref{usefuldiamond}, the lemma describes how to move in the direction of a Lie bracket of many directions, by moving repeatedly backwards and forwards in the individual directions. Recall that an iterated bracket is of length $k$ with entries $Y_1,\ldots, Y_q$ if it is of the form $[Y_{j_1},\ldots[Y_{j_{k-1}}, Y_{j_k}]\ldots]$ for some multi-index $(j_1,\ldots, j_k)\in \{1,\ldots, q\}^k$.

\begin{lemma}\label{diamondlemma}
Suppose $Y_{1}, \ldots, Y_{q} \in \mathfrak{g}$. Then we can write 
\[ [Y_{q}, \ldots [Y_{2}, Y_{1}]\ldots]=Y_{j_{1}}\diamond \cdots \diamond Y_{j_{c(q)}}+R(Y_{1}, \ldots, Y_{q}),\]
where $c(q)$ is an integer depending only on $q$, and
\[Y_{j_{1}}, \ldots, Y_{j_{c(q)}} \in \{ \pm Y_{1}, \ldots, \pm Y_{q}\}.\]
The remainder term $R(Y_{1}, \ldots, Y_{q})$ is a linear combination of brackets of length $\geq q+1$ of the vector fields $Y_{1}, \ldots, Y_{q}$.
\end{lemma}

The next lemma shows how to move in a direction $U+V$ by repeatedly making increments in directions $U$ and $V$. Such a statement will be useful when we want to relate directional derivatives in direction $U+V$ to directional derivatives in direction $U$ and in direction $V$.

\begin{lemma}\label{ballboxtype}
There is a constant $C=C_{2}>0$ for which the following holds. For any pair $U, V\in V_{1}$, there exist $U_{1}, \ldots , U_{N} \in \{U, V\}$ and $\rho_{1}, \ldots, \rho_{N}\in \mathbb{R}$ with $|\rho_{i}|\leq C$ and $N\leq C$ such that
\[ \exp(U+V)=\exp(\rho_{1} U_{1})\exp(\rho_{2} U_{2})\cdots \exp(\rho_{N} U_{N}).\]
\end{lemma}

\begin{proof}
We prove the lemma by induction on the step of $\mathbb{G}$. Throughout the proof, $C$ denotes a constant depending only on $\mathbb{G}$, which may vary from line to line. If $\mathbb{G}$ has step one the statement is clear: simply take $N=2$, $U_{1}=U$, $U_{2}=V$ and use the equality:
\[\exp(U+V)=\exp(U)\exp(V),\]
which follows from the Baker-Campbell-Hausdorff formula \eqref{usefuldiamond} and \eqref{BCH}.

Suppose $s>1$ and the lemma holds for Carnot groups of step $s-1$. Let $\mathbb{G}$ be a Carnot group of step $s$ with Lie algebra $\mathfrak{g}=V_{1}\oplus \cdots \oplus V_{s}$ and adapted basis $X_{1}, \ldots, X_{n}$ of $\mathfrak{g}$. Let $q=h_{s-1}$, then $X_{1}, \ldots, X_{q}$ is a basis of $V_{1}\oplus \cdots \oplus V_{s-1}$ and $X_{q+1}, \ldots, X_{n}$ is a basis of $V_{s}$. Define $\mathbb{H}$ as the quotient group of $\bbG$ modulo $\exp(V_s)$. Then $\mathbb{H}$ is a Carnot group of step $s-1$. Let $F\colon \mathbb{G}\to \mathbb{H}$ be the quotient mapping and define $Y_i:=dF(X_i)$ for $1\leq i\leq q$.
Then $Y_{1}, \ldots, Y_{q}$ is an adapted basis of the Lie algebra $\mathfrak{H}=W_{1}\oplus \cdots \oplus W_{s-1}$ of $\mathbb{H}$. The Lie bracket is given by:
\begin{equation}\label{newbrackets} [Y_{i},Y_{j}]_{\mathbb{H}}=\sum_{k=1}^{q}a_{k}Y_{k} \qquad \mbox{if} \qquad [X_{i},X_{j}]_{\mathbb{G}}=\sum_{k=1}^{n}a_{k}X_{k}.\end{equation}

Suppose $U, V\in V_{1}$. Then $dF(U), dF(V)\in W_{1}$. Since $\mathbb{H}$ has step $s-1$, we may apply our inductive hypothesis to $\mathbb{H}$ and use linearity of $dF$ to write:
\begin{equation}\label{hypothesis} \exp_{\mathbb{H}}(dF(U+V))=\exp_{\mathbb{H}}(\rho_{1}dF(U_1))\cdots \exp_{\mathbb{H}}(\rho_{N}dF(U_N)),\end{equation}
with $U_{i}\in \{ U, V\}$, $|\rho_{i}|\leq C$ and $N\leq C$. Here $C$ is a constant depending only on $\mathbb{H}$, hence $C$ is determined by $\mathbb{G}$. Since $\exp_{\mathbb{H}}\circ \, dF = F\circ \exp_{\mathbb{G}}$, \eqref{hypothesis} implies:
\begin{align}
F(\exp_{\mathbb{G}}(U+V))&=F(\exp_{\mathbb{G}}(\rho_{1}U_1))\cdots F(\exp_{\mathbb{G}}(\rho_{N}U_N))\nonumber \\
&=F(  \exp_{\mathbb{G}}(\rho_{1}U_1)\cdots \exp_{\mathbb{G}}(\rho_{N}U_N)).\label{touse1}
\end{align}

\begin{claim}\label{formforZ}
There exists $Z\in V_{s}$ such that
\begin{equation}\label{someZ}  \exp_{\mathbb{G}}(U+V+Z) = \exp_{\mathbb{G}}(\rho_{1}U_1)\cdots \exp_{\mathbb{G}}(\rho_{N}U_N).\end{equation}
Furthermore, $Z=\eta_{1}Z_{1}+ \ldots +\eta_{p}Z_{p}$ with $|\eta_{i}|\leq C$ and $p\leq C$ for a constant $C$ determined by $\mathbb{G}$. Each term $Z_{i}$ is a Lie bracket of length $s$ of $U$ and $V$.
\end{claim}

\begin{proof}
The map $\exp_{\mathbb{G}}$ is a diffeomorphism, so necessarily there exists $Z\in \mathfrak{g}$ such that \eqref{someZ} holds. Combining \eqref{touse1} and \eqref{someZ} gives:
\[ F(\exp_{\mathbb{G}}(U+V+Z))= F(\exp_{\mathbb{G}}(U+V)).\]
Hence:
\[ \exp_{\mathbb{H}} (dF(U)+dF(V)+dF(Z))= \exp_{\mathbb{H}} (dF(U)+dF(V)).\]
Since $\exp_{\mathbb{H}}$ is a diffeomorphism this implies 
\[ dF(U)+dF(V)+dF(Z) = dF(U)+dF(V),\]
so $dF(Z)=0$. The kernel of $dF$ is $V_{s}$, so $Z\in V_{s}$. 
By \eqref{someZ}, the fact that $\exp_{\mathbb{G}}$ is invertible and \eqref{usefuldiamond} we obtain:
\begin{align}
Z&=\exp_{\bbG}^{-1}(\exp_{\mathbb{G}}(\rho_{1}U_1)\cdots \exp_{\mathbb{G}}(\rho_{N}U_N))-U-V\\
\nonumber
&=(\rho_1 U_1) \diamond \cdots \diamond (\rho_N U_N) - U - V.
\end{align}
Using the Baker-Campbell-Hausdorff formula \eqref{BCH} we can write:
\begin{align}
Z=J_{1}+\ldots+J_{s},
\end{align}
where each $J_i$ is a linear combination of Lie brackets of length $i$ of $U$ and $V$. The number of terms in each linear combination is bounded by a constant depending on $\mathbb{G}$. The scalars in these linear combinations are polynomials in the numbers $\rho_i$, with bounded degree and coefficients. Since the numbers $\rho_{i}$ are uniformly bounded, the scalars are also bounded by a constant depending on $\mathbb{G}$. Note $Z\in V_{s}$ implies $J_i=0$ for every $1\leq i\leq s-1$. Hence $Z=J_{s}$ is of the desired form.
\end{proof}

Any of the $Z_{i}$ in Claim \ref{formforZ} can be written in the form:
\begin{equation}\label{formstoadd} [E_{s}, \ldots [E_{2}, E_{1}]\ldots]_{\mathbb{G}} \in V_{s}, \qquad E_{i} \in \{ U, \, V\}.\end{equation}
Given such a Lie bracket, use Lemma \ref{diamondlemma} to find $j_{1}, \ldots, j_{c(s)}$ such that
\[ [E_{s}, \ldots [E_{2}, E_{1}]\ldots]_{\mathbb{G}} = E_{j_{1}}\diamond \cdots \diamond E_{j_{c(s)}}+R(E_{1}, \ldots, E_{s}),\]
where $c(s)$ is an integer depending only on $s$, and
\[E_{j_{1}}, \ldots, E_{j_{c(s)}} \in \{ \pm U, \pm V \}.\]
The remainder term $R(E_{1}, \ldots, E_{s})$ is a linear combination of brackets of height $\geq s+1$ of the vector fields $E_{1}, \ldots, E_{s}$. Since $\mathbb{G}$ has step $s$, this implies $R=0$. Consequently:
\[ [E_{s}, \ldots [E_{2}, E_{1}]\ldots]_{\mathbb{G}} = E_{j_{1}}\diamond \cdots \diamond E_{j_{c(s)}}.\]
Using \eqref{usefuldiamond}, this implies:
\begin{equation}\label{onebracketgood} \exp_{\mathbb{G}}([E_{s}, \ldots [E_{2}, E_{1}]\ldots]_{\mathbb{G}} ) = \exp_{\mathbb{G}}(E_{j_{1}}) \cdots \exp_{\mathbb{G}}(E_{j_{c(s)}}),\end{equation}
where $c(s)\leq C$. 

Claim \ref{formforZ} states that $Z=\eta_{1}Z_{1}+ \ldots +\eta_{p}Z_{p}$ with $|\eta_{i}|\leq C$, $p \leq C$ and each $Z_{i}$ of the form in \eqref{formstoadd}. Since $Z_{i}\in V_{s}$ for $1\leq i\leq p$ we have $[Z_{i},Z_{j}]_{\mathbb{G}}=0$, so:
\[\exp_{\mathbb{G}}(Z)=\exp_{\mathbb{G}}(\eta_{1} Z_1)\cdots \exp_{\mathbb{G}}(\eta_{p}Z_p).\]
Hence we may use \eqref{onebracketgood} for each $Z_{i}$ and combine the resulting expressions to write:
\begin{equation}\label{Zdecomp} \exp_{\mathbb{G}}(Z)= \exp_{\mathbb{G}}(\delta_{1}Q_{1})\cdots \exp(\delta_{M}Q_{M}),\end{equation}
with $M\leq C$, $|\delta_{i}|\leq C$ and $Q_{i}\in \{ U, V\}$. Combining \eqref{someZ} and \eqref{Zdecomp}, we deduce:
\begin{align}\label{repUplusV}
& \exp_{\mathbb{G}}(\rho_{1}U_1)\cdots \exp_{\mathbb{G}}(\rho_{N}U_N)  \exp_{\mathbb{G}}(-\delta_{M}Q_{M})\cdots \exp_{\mathbb{G}}(-\delta_{1}Q_{1}) \nonumber \\
&\qquad = \exp_{\mathbb{G}}(U+V+Z)\exp_{\mathbb{G}}(-Z)\nonumber \\
&\qquad =\exp_{\mathbb{G}}(U+V).
\end{align}
The second equality follows from the fact $Z\in V_{s}$ so $[U+V+Z,-Z]_{\mathbb{G}}=0$ and the Baker-Campbell-Hausdorff formula \eqref{BCH} simplifies. Notice that $|\rho_{i}|, |\delta_{i}|, N,M \leq C$ and $U_{i}, Q_{i} \in \{ U, V\}$, so \eqref{repUplusV} gives the required representation holds for $\exp(U+V)$. Hence the lemma holds for $\mathbb{G}$ which has step $s$. This completes the inductive step, proving the lemma.
\end{proof}

\subsection{Geometry to differentiability}

We now use Lemma \ref{distances} and Lemma \ref{ballboxtype} to show directional derivatives act linearly outside a $\sigma$-porous set.

\begin{theorem}\label{porosityderivatives}
Suppose $f\colon \mathbb{G} \to \mathbb{R}$ is Lipschitz. Then there is a $\sigma$-porous set $A\subset \mathbb{G}$ such that directional derivatives act linearly at every point $x\in \mathbb{G} \setminus A$, namely the following implication holds:
if $E_{1}f(x)$ and $E_{2}f(x)$ exist for some $E_{1}, E_{2}\in V_{1}$, then $(a_{1}E_{1}+a_{2}E_{2})f(x)$ exists and
\[(a_{1}E_{1}+a_{2}E_{2})f(x)=a_{1}E_{1}f(x)+a_{2}E_{2}f(x)\]
for all $a_{1}, a_{2}\in \mathbb{R}$.
\end{theorem}

It follows directly from the definition that if a directional derivative $Ef(x)$ exists then $(sE)f(x)$ exists and is equal to $s(E(f(x)))$ for any $s\in \mathbb{R}$. Hence to prove Theorem \ref{porosityderivatives} we may assume $a_{1}=a_{2}=1$.

Recall the constant $C_{2}$ from Lemma \ref{ballboxtype}. Given $U, V\in V_{1}\setminus \{0\}$, $y, z\in \mathbb{R}$ and $\varepsilon, \delta>0$, let $A(U,V,y,z,\varepsilon,\delta)$ be the set of $x\in \mathbb{G}$ such that for all $|t|<\delta$:
\begin{equation}\label{Ugood} |f(x\exp(tU))-f(x)-ty|\leq \varepsilon |t|,\end{equation}
\begin{equation}\label{Vgood} |f(x\exp(tV))-f(x)-tz|\leq \varepsilon |t|,\end{equation}
and there exist arbitrarily small $t$ for which:
\begin{equation}\label{UVbad}|f(x\exp(tU+tV))-f(x)-t(y+z)|> 2\varepsilon C_{2}|t|.\end{equation}

Fix $x\in A(U,V,y,z,\varepsilon,\delta)$. Use Lemma \ref{ballboxtype} to choose $U_{1}, \ldots, U_{N} \in \{ U, V\}$ and $\rho_{1}, \ldots, \rho_{N}\in \mathbb{R}$ with $|\rho_{i}|\leq C_{2}$ and $N\leq C_{2}$ such that
\begin{equation}\label{composition}
\exp(U+V)=\exp(\rho_{1}U_{1})\cdots \exp(\rho_{N}U_{N}).
\end{equation}
Choose 
\begin{equation}\label{tsmall}
|t|<\min \left( \delta,\, \frac{1}{(1+\omega(U)+\omega(V))C_{2}}\right)
\end{equation}
satisfying \eqref{UVbad}. Dilating both sides of \eqref{composition} by the factor $t$ gives:
\begin{equation}\label{compositiont} 
\exp(tU+tV)=\exp(t\rho_{1}U_{1})\cdots \exp(t\rho_{N}U_{N}).
\end{equation}

Let $x_{0}=x$. For $1\leq i\leq N$ let $x_{i}=x_{i-1}\exp (t\rho_{i}U_{i})$, $y_{i}=y$ if $U_{i}=U$ and $y_{i}=z$ if $U_{i}=V$. Note that $x_{N}=x\exp (tU+tV)$ by \eqref{compositiont}. 

\begin{claim}\label{claim1}
If $U$ and $V$ are linearly independent then the following statements hold:
\begin{enumerate}
\item $\sum_{i=1}^{N}\rho_{i}y_{i} =y+z$,
\item there exists $0\leq i\leq N-1$ such that
\[|f(x_{i+1})-f(x_{i})-t\rho_{i+1}y_{i+1}|>2\varepsilon |t|.\]
\end{enumerate}
\end{claim}

\begin{proof}
We first verify (1). Define:
\[I_{U}=\{i: U_{i}=U\} \mbox{ and } I_{V}=\{i: U_{i}=V\}.\]
Recall that $p\colon \mathbb{G}\to \mathbb{R}^{m}$ is the projection onto the first $m$ coordinates. The equality \eqref{composition} implies
\begin{align*}
p(U)+p(V)&=\rho_{1}p(U_{1})+\ldots +\rho_{N}p(U_{N})\\
&= \Big(\sum_{i\in I_{U}} \rho_{i}\Big) p(U) + \Big(\sum_{i\in I_{V}} \rho_{i}\Big) p(V).
\end{align*}
Since $U, V\in V_{1}$ and are linearly independent, $p(U)$ and $p(V)$ are linearly independent. Hence:
\[  \sum_{i\in I_{U}} \rho_{i} =  \sum_{i\in I_{V}} \rho_{i}   =1.\]
Consequently:
\begin{align*}
\sum_{i=1}^{N}\rho_{i}y_{i}&= \Big(\sum_{i\in I_{U}} \rho_{i}\Big)y + \Big(\sum_{i\in I_{V}} \rho_{i}\Big)z\\
&=y+z,
\end{align*}
which proves (1).

Next suppose (2) fails. Then:
\begin{equation}\label{dontforget}
|f(x_{i+1})-f(x_{i})-t\rho_{i+1}y_{i+1}|\leq 2\varepsilon |t| \mbox{ for }0\leq i\leq N-1.
\end{equation}
We estimate using (1), \eqref{dontforget} and $N\leq C_{2}$:
\begin{align*}
| f(x\exp(tU+tV))-f(x)-t(y+z) | &= \left| f(x_{N})-f(x)-t\sum_{i=1}^{N} \rho_{i}y_{i}\right|\\
&=\left| \sum_{i=1}^{N} (f(x_{i})-f(x_{i-1})-t\rho_{i}y_{i}) \right|\\
&\leq \sum_{i=1}^{N} |f(x_{i})-f(x_{i-1})-t\rho_{i}y_{i}|\\
&\leq 2\varepsilon C_{2}|t|,
\end{align*}
This estimate violates our choice of $t$ satisfying \eqref{UVbad}. This proves (2).
\end{proof}

\begin{claim}\label{claim2}
The set $A(U,V,y,z,\varepsilon,\delta)$ is porous for each choice of the parameters $U, V, y, z, \varepsilon, \delta$.
\end{claim}

\begin{proof}
We may assume that $U$ and $V$ are linearly independent, otherwise $A(U,V,y,z,\varepsilon,\delta)=\varnothing$ and the claim is trivial. Use Claim \ref{claim1}(2) and the equality $x_{i+1}=x_{i}\exp(t\rho_{i+1}U_{i+1})$ to choose $0\leq i\leq N-1$ satisfying:
\begin{equation}\label{touse} |f(x_{i}\exp(t\rho_{i+1}U_{i+1}))-f(x_{i})-t\rho_{i+1}y_{i+1}|>2\varepsilon |t|.\end{equation}
Recall the constant $C_{1}>0$ defined in Lemma \ref{distances}. Let
\[\Lambda= 2\mathrm{Lip}(f)C_{1}C_{2}\max(\omega(U),\omega(V)),\]
\begin{equation}\label{deflambda}
\lambda=\min \left( \frac{\varepsilon^{s}}{\Lambda^s}, \,  \frac{\varepsilon}{2\mathrm{Lip}(f)}, \, 1\right),
\end{equation}
and
\begin{equation}\label{defr}
r= \min (\lambda \omega(U),\, \lambda \omega(V),\, \lambda) |t|.
\end{equation}
Fix $q\in B(x_{i}, r)$. We will apply Lemma \ref{distances} with parameters:
\begin{itemize}
\item $\lambda$ as defined in \eqref{deflambda},
\item $t$ replaced by $t \rho_{i+1} \omega(U_{i+1})$,
\item $x$ replaced by $x_{i}$,
\item $y$ replaced by $q$,
\item $U$ replaced by $U_{i+1}/\omega(U_{i+1})$.
\end{itemize}
To see the hypotheses of Lemma \ref{distances} are satisfied, notice:
\begin{itemize}
\item \eqref{tsmall} and the bound $|\rho_{i+1}|\leq C_{2}$ implies $|t \rho_{i+1} \omega(U_{i+1})|<1$,
\item $0<\lambda<1$ is clear from the definition in \eqref{deflambda},
\item $d(x_{i},q)\leq \lambda |t|$ using $q\in B(x_{i}, r)$ and the definition of $r$ in \eqref{defr}.
\end{itemize}
Also note $\omega(U_{i+1}/\omega(U_{i+1}))\leq 1$ and recall $|\rho_{i+1}|\leq C_{2}$. Applying Lemma \ref{distances} gives:
\begin{align*}
d(x_{i}\exp (t\rho_{i+1}U_{i+1}), q\exp (t\rho_{i+1}U_{i+1})) &\leq C_{1}\lambda^{1/s}|t \rho_{i+1}|  \omega(U_{i+1}) \\
&\leq C_{1}C_{2}|t|\omega(U_{i+1})\lambda^{1/s} \\
&\leq C_{1}C_{2}|t|\omega(U_{i+1})\varepsilon / \Lambda \\
&\leq \varepsilon |t|/(2\mathrm{Lip}(f)).
\end{align*}
Hence:
\begin{equation}\label{firstone}
|f(x_{i}\exp(t\rho_{i+1}U_{i+1})) - f(q\exp(t\rho_{i+1}U_{i+1}))|\leq \varepsilon |t|/2.
\end{equation}
Since $q\in B(x_i, r)$ we can also estimate:
\begin{equation}\label{secondone}
|f(x_{i})-f(q)| \leq \mathrm{Lip}(f) r \leq \mathrm{Lip}(f) \lambda |t| \leq \varepsilon |t|/2.
\end{equation}
Applying \eqref{touse} together with \eqref{firstone} and \eqref{secondone} gives:
\begin{equation}\label{mar8} |f(q\exp(t\rho_{i+1}U_{i+1}))-f(q)-t\rho_{i+1}y_{i+1}|>\varepsilon |t|.\end{equation}
Recall that either $U_{i+1}=U$ and $y_{i+1}=y$ or $U_{i+1}=V$ and $y_{i+1}=z$. Hence \eqref{mar8} shows that \eqref{Ugood} or \eqref{Vgood} fails with $x$ replaced by $q$. Consequently $q\notin A(U,V,y,z,\varepsilon,\delta)$. Since $q$ was an arbitrary member of $B(x_{i}, r)$, we deduce:
\[B(x_{i}, r)\cap A(U,V,y,z,\varepsilon,\delta)=\varnothing.\]
Since $N\leq C_{2}$ and $|\rho_{i}|\leq C_{2}$, we can estimate as follows: 
\begin{align*}
d(x,x_{i})&\leq d(x_{0}, x_{1})+ \ldots + d(x_{i-1},x_{i})\\
&\leq  \omega(t\rho_{1}U_{1}) + \ldots + \omega(t \rho_{N}U_{N})\\
&\leq  |t||\rho_{1}| \omega(U_{1}) + \ldots + |t||\rho_{N}|\omega(U_{N})\\
&\leq C_{2}^{2}\max (\omega(U), \omega(V)) |t|.
\end{align*}
Since $r=\min (\lambda \omega(U),\, \lambda \omega(V),\, \lambda) |t|$, $B(x_{i}, r)$ is a relatively large hole in $A(U,V,y,z,\varepsilon,\delta)$ close to $x$. Since $t$ can be chosen arbitrarily small, it follows that $A(U,V,y,z,\varepsilon,\delta)$ is porous at $x$. The point $x$ was chosen arbitrarily from $A(U,V,y,z,\varepsilon,\delta)$, so the claim is proven.
\end{proof}

\begin{proof}[Proof of Theorem \ref{porosityderivatives}]
Fix a countable dense set $W\subset V_{1}$. Let $A$ be the countable union of sets $A(U,V,y,z,\varepsilon,\delta)$, with $U, V \in W$ and $y, z, \varepsilon, \delta \in \mathbb{Q}$ such that $\varepsilon, \delta>0$. Clearly $A$ is $\sigma$-porous.

Suppose $x\in \mathbb{G}\setminus A$ and $E_{1}f(x)$, $E_{2}f(x)$ exist. Choose rational $\varepsilon>0$ and rational $\delta>0$ such that for all $|t|<\delta$:
\[|f(x\exp(tE_{1}))-f(x)-tE_{1}f(x)|\leq \varepsilon |t|/2\]
and
\[|f(x\exp(tE_{2}))-f(x)-tE_{2}f(x)|\leq \varepsilon |t|/2.\]
Fix $y, z\in \mathbb{Q}$ such that $|y-E_{1}f(x)|< \varepsilon/4$ and $|z-E_{2}f(x)|<\varepsilon/4$. Choose $F_{1}, F_{2}\in W$ such that
\begin{align*}
d(\exp(E_{1}),\exp(F_{1})) &<\varepsilon/4\mathrm{Lip}(f),\\
d(\exp(E_{2}),\exp(F_{2})) &<\varepsilon/4\mathrm{Lip}(f),
\end{align*}
and
\[d(\exp(E_{1}+E_{2}),\exp(F_{1}+F_{2})) \leq \varepsilon.\]

Analogues of \eqref{Ugood} and \eqref{Vgood} hold with $E_{1}$ replaced by $F_{1}$ and $E_{2}$ replaced by $F_{2}$. For the first one notice that for any $|t|<\delta$:
\begin{align*}
|f(x\exp tF_{1})-f(x)-ty| &\leq |f(x\exp tE_{1})-f(x)-tE_{1}f(x)|\\
&\qquad + |f(x\exp tE_{1})-f(x\exp tF_{1})|\\
&\qquad +|ty-tE_{1}f(x)|\\
&\leq \varepsilon |t|/2 + \mathrm{Lip}(f)|t|d(\exp(E_{1}),\exp(F_{1}))\\
&\qquad +|t| |y-E_{1}f(x)|\\
&\leq \varepsilon|t|.
\end{align*} 
The second estimate is similar. Since $x\in \mathbb{G}\setminus A$, necessarily $x \notin A(F_{1},F_{2},y,z,\varepsilon,\delta)$. However, \eqref{Ugood} and \eqref{Vgood} hold, so necessarily \eqref{UVbad} fails. Hence, for all sufficiently small $t$,
\[|f(x\exp(tF_{1}+tF_{2}))-f(x)-t(y+z)|\leq 2\varepsilon C_{1}|t|.\]
This implies that for sufficiently small $t$:
\begin{align*}
& |f(x\exp(tE_{1}+tE_{2}))-f(x)-t(E_{1}f(x)+E_{2}f(x))|\\
& \leq |f(x\exp(tF_{1}+tF_{2}))-f(x)-t(y+z)|\\
&\qquad + |t|\mathrm{Lip}(f)d(\exp(E_{1}+E_{2}),\exp(F_{1}+F_{2}))\\
&\qquad + |t| |E_{1}f(x)-y| + |t| |E_{2}f(x)-z|\\
&\leq 2\varepsilon C_{1}|t| + |t|\mathrm{Lip}(f)\varepsilon + |t|\varepsilon/2\\
&\leq (2C_{1}+\mathrm{Lip}(f)+1/2)\varepsilon |t|.
\end{align*}
This implies that $(E_{1}+E_{2})f(x)$ exists and is equal to $E_{1}f(x)+E_{2}f(x)$, which proves the theorem.
\end{proof}

\section{Regularity and differentiability}

For Lipschitz maps between Euclidean spaces, existence and linear action of directional derivatives implies differentiability (by definition). The following example shows that for a Lipschitz map $f\colon \mathbb{G}\to \mathbb{R}$, existence and linear action of directional derivatives (in only horizontal directions) does not suffice for Pansu differentiability.

\begin{example}\label{linearnotdiff}
Let $\mathbb{H}^{1}\equiv \mathbb{R}^3$ be the first Heisenberg group. The Heisenberg group is the simplest example of non trivial Carnot group of step $2$. The group law has the form
\[
(x,y,t)(x',y',t')=\Big(x+x', y+y', t+t'+\frac{1}{2}(xy'-x'y)\Big).
\]
The homogeneous norm defined in \eqref{auxnorm} reads as
\[ \| (x,y,t)\| = ((x^2+y^2)^2 + t^2)^{1/4},\quad (x,y,t)\in \mathbb{H}^1\]
and $\tilde{d}(a,b)=\|a^{-1}b\|$ defines a metric on $\mathbb{H}^{1}$.
Let
\[ A= (\mathbb{R}^{2}\times \{0\})\cup (\{(0,0)\}\times \mathbb{R})\subset \mathbb{H}^{1},\]
equipped with the restriction of the metric $d$ on $\mathbb{H}^{1}$. Define $f\colon A\to \mathbb{R}$ by $f(x,y,0)=0$ and $f(0,0,t)=\sqrt{|t|}$. We claim that $f$ is Lipschitz with respect to restriction of the metric $\tilde{d}$. Clearly, it suffices to prove that there exists $C>0$ such that
\begin{equation}\label{ineq1}
|f(x,y,0)-f(0,0,t)|\leq C \tilde{d}((x,y,0),(0,0,t))\quad \mbox{ for all } x,y,t\in\mathbb{R}
\end{equation}
and
\begin{equation}\label{ineq2}
|f(0,0,t)-f(0,0,s)|\leq C \tilde{d}((0,0,t),(0,0,s))\quad \mbox{ for all } t,s\in\mathbb{R}
\end{equation}
Using the definition of the group law and $\tilde{d}$:
\begin{align*}
\tilde{d}((x,y,0),(0,0,t)) =\|(-x,-y,t)\| \geq c \big(\|(x,y)\|_{\mathbb{R}^2}+ \sqrt{|t|}\big)
\end{align*}
and
\[|f(x,y,0)-f(0,0,t)|= \sqrt{|t|} \leq \frac{1}{c} \tilde{d}((x,y,0),(0,0,t))\]
so \eqref{ineq1} follows.
Moreover,
\[
|f(0,0,t)-f(0,0,s)|=\left|\sqrt{|t|}-\sqrt{|s|}\right| \leq \sqrt{|t-s|}= \tilde{d}((0,0,t),(0,0,s)).
\]
By the classical McShane extension theorem (see for example \cite[Theorem 6.2]{Hei01}), $f$ admits an extension to a Lipschitz function $f\colon \mathbb{H}^{1}\to \mathbb{R}$, note that by \eqref{equivdist}, $f$ is Lipschitz also with respect to the Carnot-Carath\'eodory distance $d$. Clearly $Ef(0,0,0)=0$ for all $E\in V_{1}$, so $f$ has all directional derivatives and they act linearly at $(0,0,0)$. Indeed, recalling that $\exp(tE)\in \mathbb{R}^2\times\{0\}$, we have:
\[
Ef(0,0,0)=\lim_{t\to 0}\frac{f(\exp(tE))-f(0,0,0)}{t}=0.
\]
However, $f$ is not Pansu differentiable at $(0,0,0)$. Indeed:
\[ \lim_{t \to 0} \frac{|f(0,0,t^2)-f(0,0,0)|}{d((0,0,t^2),(0,0,0))} = \lim_{t \to 0} \frac{|t|}{|t|}=1>0.\]
\end{example}

In this section we adapt \cite{LP03} to define regular points of Lipschitz functions on Carnot groups and show that the set of irregular points of a Lipschitz function is $\sigma$-porous (Proposition \ref{irregularporous}). We then show that outside another $\sigma$-porous set, existence of directional derivatives (in horizontal directions) implies Pansu differentiability (Theorem \ref{differentiableatlast} and Corollary \ref{directionaltodiff}). This then yields a new proof of Pansu's theorem (Corollary \ref{Pansucor}).

\subsection{Regular points}

We adapt \cite[Definition 3.1]{LP03} from Banach spaces to define regular points for Lipschitz maps on Carnot groups. Intuitively, $x$ is a regular point of $f$ if whenever $E\in V_{1}$ and $Ef(x)$ exists, then $Ef(x)$ controls changes in $f$ along all lines close to $x$ in direction $E$.

\begin{definition}\label{defregularpoint}
Suppose $f\colon \mathbb{G}\to \mathbb{R}$. We say that $x\in \mathbb{G}$ is a \emph{regular point} of $f$ if for every $E\in V_{1}$ for which $Ef(x)$ exists, 
\[\lim_{t\to 0} \frac{f(x\delta_{t}(u)\exp(tE))-f(x\delta_{t}(u))}{t}=Ef(x)\]
uniformly for $d(u)\leq 1$. A point is \emph{irregular} if it is not regular.
\end{definition}

We adapt \cite[Proposition 3.3]{LP03} from Banach spaces to show that irregular points of a Lipschitz function form a $\sigma$-porous set.

\begin{proposition}\label{irregularporous}
Let $f\colon \mathbb{G}\to \mathbb{R}$ be Lipschitz. Then the set of irregular points of $f$ is $\sigma$-porous.
\end{proposition}

\begin{proof}
Let $\mathcal{J}$ be a countable dense subset of $V_{1}$. For $p,q\in \mathbb{N}$, $E \in \mathcal{J}$ and $w\in \mathbb{Q}$, let $A_{p,q,E,w}$ be the set of $x\in \mathbb{G}$ such that
\begin{equation}\label{tocontradict} |f(x\exp tE)-f(x)-tw| \leq |t|/p \mbox{ for}\  |t|<1/q\end{equation}
and
\[\limsup_{t\to 0} \sup_{d(u)\leq 1} \left| \frac{f(x\delta_{t}(u)\exp(tE))-f(x\delta_{t}(u))}{t} - w\right| > 2/p.\]
Fix $p,q,E,w$ and $x\in A_{p,q,E,w}$. Then there are arbitrarily small $|t|<1/q$ such that for some $u\in \mathbb{G}$ with $d(u)\leq 1$, depending on $t$,
\begin{equation}\label{mar8pm}|f(x\delta_{t}(u)\exp(tE))-f(x\delta_{t}(u))-tw|>2|t|/p.\end{equation}
Let
\[\lambda=\min \left( \frac{1}{(2C_{1}p\mathrm{Lip}(f)(1+\omega(E)))^{s}},\, \frac{1}{2p\mathrm{Lip}(f)},\, 1\right),\]
where $C_{1}$ is the constant from Lemma \ref{distances}. Fix $|t|<1/q$ and $u\in \mathbb{G}$ with $d(u)\leq 1$ satisfying \eqref{mar8pm}. Suppose $y\in B(x\delta_{t}(u),\lambda|t|)$. We may apply Lemma \ref{distances}, with $U$ replaced by $E$ and $x$ replaced by $x\delta_{t}(u)$, to obtain:
\begin{align*}
|f(y\exp(tE))-f(x\delta_{t}(u)\exp(tE))| &\leq \mathrm{Lip}(f)d(y\exp(tE), x\delta_{t}(u)\exp(tE))\\
&\leq \mathrm{Lip}(f) C_{1} \lambda^{1/s}|t| \max (1, \omega(E))\\
&\leq |t|/2p.
\end{align*}
Hence, using also \eqref{mar8pm} and the assumption $d(y,x\delta_{t}(u))<\lambda |t|$,
\begin{align*}
|f(y\exp(tE))-f(y)-tw| &\geq |f(x\delta_{t}(u)\exp(tE))-f(x\delta_{t}(u))-tw|\\
&\qquad -|f(y\exp(tE))-f(x\delta_t(u)\exp(tE))|\\
&\qquad \qquad -|f(x\delta_t(u))-f(y)|\\ 
&> |t|/p.
\end{align*}
Since $|t|<1/q$ the previous estimate contradicts \eqref{tocontradict}, so $y\notin A_{p,q,E,w}$. To summarise:
\[d(x,x\delta_{t}(u))=d(\delta_{t}(u))\leq |t|\]
and we have shown:
\[B(x\delta_{t}(u),\lambda|t|)\cap A_{p,q,E,w}=\varnothing.\]
Since $|t|$ could be chosen arbitrarily small, this shows $A_{p,q,E,w}$ is porous. Every irregular point of $f$ belongs to one of the countable collection of porous sets $A_{p,q,E,w}$, where $p,q\in \mathbb{N}$, $E \in \mathcal{J}$ and $w\in \mathbb{Q}$. Hence the set of irregular points of $f$ is $\sigma$-porous, as desired.
\end{proof}

\subsection{Differentiability and Pansu's theorem}

To pass from directional derivatives to Pansu's theorem, we use a lemma showing how to join arbitrary points by following directions from our chosen basis $X_{1}, \ldots, X_{m}$ of $V_{1}$. To prove our desired statement we first quote the following lemma \cite[Theorem 19.2.1]{BLU}.

\begin{lemma}\label{magiclemma}
Let $\mathcal{Z}=\{Z_{1}, \ldots, Z_{m}\}$ be a basis for $V_{1}$ and fix a homogeneous norm $\rho$ on $\mathbb{G}$. Then there exist constants $M\in \mathbb{N}$ (depending only on $\mathbb{G}$) and $c_{0}>0$ (depending on $\mathbb{G}$, $\rho$ and $\mathcal{Z}$) for which the following holds.

For every $x\in \mathbb{G}$, there exist $x_{1}, \ldots, x_{M}\in \exp(V_{1})$ with the following properties:
\begin{itemize}
\item $x=x_{1}\cdots x_{M}$,
\item $\rho(x_{j})\leq c_{0}\rho(x)$ for all $j=1, \ldots, M$,
\item for every $j=1,\ldots, M$, there exist $t_{j}\in \mathbb{R}$ and $i_{j}\in \{1, \ldots, m\}$ such that $x_{j}=\exp(t_{j}Z_{i_{j}})$.
\end{itemize}
\end{lemma}

\begin{lemma}\label{pathlemma}
There are constants $M\in \mathbb{N}$ and $Q>0$ such that the following holds for every point $h\in \mathbb{G}$. For $1\leq j\leq M$, there exist $t_{j} \geq 0$ and $E_{j}\in \{\pm X_{1}, \ldots, \pm X_{m}\}$ such that:
\begin{enumerate}
\item $\sum_{j=1}^{M} t_{j} \leq Qd(h)$,
\item $h=\exp(t_{1}E_{1}) \cdots \exp(t_{M}E_{M})$.
\end{enumerate}
\end{lemma}

\begin{proof}
Apply Lemma \ref{magiclemma} with the basis $\mathcal{Z}=\{X_{1}, \ldots, X_{m}\}$ of $V_{1}$ and homogeneous norm $x\mapsto d(x)$ on $\mathbb{G}$. We obtain $M\in \mathbb{N}$ such that for every $h\in \mathbb{G}$, there exist $t_{j}\in \mathbb{R}$ and $i_{j}\in \{1, \ldots, m\}$ for $1\leq j\leq M$, such that
\begin{equation}\label{toref1}
h=\exp(t_{1}X_{i_{1}})\cdots \exp(t_{M}X_{i_{M}}),
\end{equation}
and
\begin{equation}\label{toref2}
d(\exp t_{j}X_{i_{j}})\leq c_{0}d(h).
\end{equation}
Since $\omega(X_{i_{j}})=1$ for every $j$, \eqref{toref2} implies $|t_{j}|\leq c_{0}d(h)$ for $1\leq j\leq M$. Hence:
\begin{equation}\label{toref3}
\sum_{j=1}^{M} |t_{j}|\leq Mc_{0}d(h).
\end{equation}
For each $1\leq j\leq M$ choose $E_{j}=X_{i_{j}}$ if $t_{j}\geq 0$, or choose $E_{j}=-X_{i_{j}}$ and replace $t_{j}$ by $-t_{j}$ if $t_{j}<0$. After such a replacement we have $t_{j}\geq 0$ for every $j$, \eqref{toref1} still holds giving the desired equality (2), and \eqref{toref3} gives (1).
\end{proof}

We now prove a pointwise differentiability result, whose hypotheses combine the various conditions we have investigated so far.

\begin{theorem}\label{differentiableatlast}
Suppose $f\colon \mathbb{G}\to \mathbb{R}$ is Lipschitz and $x\in \mathbb{G}$ has the following properties:
\begin{itemize}
\item $X_{1}f(x), \ldots, X_{m}f(x)$ exist,
\item directional derivatives act linearly at $x$,
\item $x$ is a regular point of $f$.
\end{itemize}
Then $f$ is Pansu differentiable at the point $x$ with group linear derivative $L(\cdot)=\langle p(\cdot), \nabla_{H}f(x) \rangle$, where $p\colon \mathbb{G}\to \mathbb{R}^{m}$ is projection onto the first $m$ coordinates.
\end{theorem}

\begin{proof}
First we remark that $h\mapsto \langle p(h), \nabla_{H}f(x) \rangle$ is group linear. Indeed, the group operation is Euclidean in the first $m$ coordinates so:
\begin{align*}
\langle p(h_{1}h_{2}), \nabla_{H}f(x) \rangle &= \langle p(h_{1})+p(h_{2}), \nabla_{H}f(x) \rangle\\
&= \langle p(h_{1}), \nabla_{H}f(x) \rangle + \langle p(h_{2}), \nabla_{H}f(x) \rangle,
\end{align*}
and
\[\langle p(\delta_{r}(h)), \nabla_{H}f(x) \rangle= \langle rp(h), \nabla_{H}f(x) \rangle = r\langle p(h), \nabla_{H}f(x) \rangle\]
for any $h_{1}, h_{2}\in \mathbb{G}$ and $r>0$.

We may assume $x=0$. Fix $M$ and $Q$ as in Lemma \ref{pathlemma} and let $\varepsilon>0$. Since $X_{1}f(x), \ldots, X_{m}f(x)$ exist and directional derivatives act linearly at $x$, necessarily $Ef(x)$ exists for every $E\in V_{1}$. Since $0$ is a regular point of $f$, for each $E\in V_{1}$ we can find $\delta>0$ such that if $0<|t|<\delta$ and $d(u)\leq 1$ then:
\begin{equation}\label{reguse}|f(\delta_{t}(u)\exp(tE))-f(\delta_{t}(u))-tEf(0)|\leq \varepsilon |t|. \end{equation}
Since $\mathcal{J}:=\{E\in V_{1}: \omega(E)\leq 1\}$ is compact, we may additionally choose $\delta>0$ so that \eqref{reguse} holds uniformly if $0<|t|<\delta$, $d(u)\leq 1$ and $E\in\mathcal{J}$.

Fix $h\in \mathbb{G}\setminus\{0\}$ with $d(h)\leq \delta/Q$. Use Lemma \ref{pathlemma} to choose $t_{i}\geq 0$ and $E_{i}\in \{\pm X_{1}, \ldots, \pm X_{m}\}$ for $1\leq i\leq M$ such that 
\begin{equation}\label{L1}
\sum_{j=1}^{M} t_{j} \leq Qd(h)
\end{equation}
and
\begin{equation}\label{L2}
h=\exp(t_{1}E_{1}) \cdots \exp(t_{M}E_{M}).
\end{equation}

Let $x_{1}=0$ and $x_{i+1}=x_{i}\exp(t_{i}E_{i})$ for $1\leq i\leq M$. Notice $x_{M+1}=h$ follows from \eqref{L2}.
\begin{claim}
For each $1\leq i\leq M$:
\[|f(x_{i+1})-f(x_{i})-t_{i}E_{i}f(0)|\leq Q\varepsilon d(h).\]
\end{claim}

\begin{proof}
Fix $1\leq i\leq M$ and let $t=Qd(h)$. We check that there exist $u\in\mathbb{G}$ with $d(u)\leq 1$ and $E\in V_{1}$ with $\omega(E)\leq 1$ such that $x_{i}=\delta_{t}(u)$ and $t_{i}E_{i}=tE$.

We begin by solving $x_{i}=\delta_{t}(u)$. Notice
\[x_{i}=\exp(t_{1}E_{1})\cdots \exp(t_{i-1}E_{i-1}).\]
Using the definition of $x_{i}$, $\omega(E_{j})\leq 1$ and \eqref{L1}, we can estimate $d(x_{i})$ as follows:
\begin{align*}
d(x_{i}) \leq \sum_{j=1}^{i-1}d(x_{j},x_{j+1})&=\sum_{j=1}^{i-1} d(\exp(t_{j}E_{j}))\\
&\leq \sum_{j=1}^{i-1} t_{j}\omega(E_{j})\\
&\leq Qd(h),
\end{align*}
Since $t=Qd(h)$, we can choose $u\in \mathbb{G}$ with $d(u)\leq 1$ such that $x_{i}=\delta_{t}(u)$.

Since $h\neq 0$, we have $t\neq 0$. Let $E=(t_{i}/t)E$. Since $\omega(E_{i})=1$ and \eqref{L1} implies $t_{i}\leq Qd(h)$, we have
\[\omega(E)= t_{i}/t=t_{i}/Qd(h)\leq 1.\]
Hence $E\in V_{1}$ solves $tE=t_{i}E_{i}$ with $\omega(E)\leq 1$ as desired.

Since:
\begin{itemize}
\item $x_{i+1}=x_{i}\exp(t_{i}E_{i})$,
\item $x_{i}=\delta_{t}(u)$ and $t_{i}E_{i}=tE$,
\item $t=Qd(h)<\delta$,
\end{itemize}
we may apply inequality \eqref{reguse} to obtain:
\[|f(x_{i+1})-f(x_{i})-t_{i}E_{i}f(0)|\leq Q\varepsilon d(h).\]
This proves the claim.
\end{proof}

The group operation in the first $m$ coordinates is Euclidean, \eqref{L2} states $\exp(t_{1}E_{1})\cdots \exp(t_{M}E_{M})=h$, and $\exp(\theta E)=\theta E(0)$ whenever $\theta\in \mathbb{R}$ and $E\in V_{1}$. Hence:
\[ (t_{1}E_{1}(0)) \cdots (t_{M}E_{M}(0))=h.\]
Since $p(X_{i}(0))$ is the $i$'th standard basis vector of $\mathbb{R}^{m}$ in coordinates, we deduce:
\[p\left( \sum_{i=1}^{M} t_{i}E_{i}(0) \right) =p\left( \sum_{i=1}^{m} h_{i}X_{i}(0) \right).\]
Hence $\sum_{i=1}^{M} t_{i}E_{i}$ and $\sum_{i=1}^{m} h_{i}X_{i}$ are two horizontal vectors whose projections in $\mathbb{R}^{m}$ are equal at $0$. Consequently:
\[\sum_{i=1}^{M} t_{i}E_{i}=\sum_{i=1}^{m}h_{i}X_{i}.\]
Since directional derivatives act linearly at $0$, we deduce:
\[\sum_{i=1}^{M} t_{i}E_{i}f(0)=\sum_{i=1}^{m}h_{i}X_{i}f(0).\]
We then estimate as follows:
\begin{align*}
|f(h) - f(0) - \langle p(h), \nabla_{H}f(0) \rangle | &= \left|\sum_{i=1}^{M} (f(x_{i+1})-f(x_{i})) - \sum_{i=1}^{m} h_{i}X_{i}f(0)\right|  \\
&= \left|\sum_{i=1}^{M} (f(x_{i+1})-f(x_{i})) - \sum_{i=1}^{M} t_{i}E_{i}f(0)\right| \\
&\leq \sum_{i=1}^{M} |f(x_{i+1})-f(x_{i}) - t_{i}E_{i}f(0)) |\\
&\leq MQ \varepsilon d(h).
\end{align*}
Hence for $d(h)\leq \delta/Q$ we have:
\[|f(h) - f(0) - \langle p(h), \nabla_{H}f(0) \rangle |\leq MQ\varepsilon d(h).\]
This proves that $f$ is Pansu differentiable at $x$ with the desired derivative. 
\end{proof}

We can combine Theorem \ref{differentiableatlast} with Theorem \ref{porosityderivatives} and Proposition \ref{irregularporous} to obtain the following application of porous sets.

\begin{corollary}\label{directionaltodiff}
Let $f\colon \mathbb{G}\to \mathbb{R}$ be Lipschitz. Then there exists a $\sigma$-porous set $P$ for which the following implication holds at all points $x\notin P$: if $X_{i}f(x)$ exists for every $1\leq i\leq m$, then $f$ is Pansu differentiable at $x$.
\end{corollary}

\begin{proof}
Use Theorem \ref{porosityderivatives} to choose a $\sigma$-porous set $A$, outside which directional derivatives act linearly. Let $B$ be the set of irregular points of $f$ defined in Definition \ref{defregularpoint}, which Proposition \ref{irregularporous} states is $\sigma$-porous. Clearly the set $A\cup B$ is $\sigma$-porous. Suppose $X_{i}f(x)$ exists for all $1\leq i\leq m$ and $x\notin A\cup B$. Then Theorem \ref{differentiableatlast} asserts that $f$ is Pansu differentiable at $x$, proving the desired implication.
\end{proof}

We will use existence of directional derivatives and Corollary \ref{directionaltodiff} to prove Pansu's theorem.

\begin{lemma}\label{directionalae}
Suppose $f\colon \mathbb{G}\to \mathbb{R}$ is Lipschitz and $E\in V_{1}$. Then the directional derivative $Ef(x)$ exists for almost every $x\in \mathbb{G}$.
\end{lemma}

\begin{proof}
There exist global coordinates in which $\bbG$ is $\bbR^{n}$, $E$ is the constant vector field $\partial_1$, and the Haar measure is Lebesgue measure \cite{BLU}. For each $a\in \bbR^{n}$, $F_a\colon \bbR \to \bbR$ defined by $F_a(t)=f(a+te_{1})$ is a composition of Lipschitz functions, hence Lipschitz so differentiable almost everywhere. Clearly $F_{a}'(t)=Ef(a+te_{1})$. Hence, for each $a\in \bbR^{n}$, $Ef(a+t e_{1})$ exists for almost every $t\in \bbR$. Using Fubini's theorem it follows that $Ef(x)$ exists for almost every $x\in \bbG$.
\end{proof}

We now use Lemma \ref{directionalae} and Corollary \ref{directionaltodiff} to obtain Pansu's theorem for mappings from $\mathbb{G}$ to a Euclidean space.

\begin{corollary}[Pansu's Theorem]\label{Pansucor}
Suppose $f\colon \mathbb{G}\to \mathbb{R}^N$ is Lipschitz. Then $f$ is Pansu differentiable almost everywhere.
\end{corollary}

\begin{proof} 
First assume $N=1$. Using Lemma \ref{directionalae}, we can find a measure zero set $A\subset \mathbb{G}$ such that $X_{i}f(x)$ exists for $1\leq i\leq m$ and $x\in \mathbb{G}\setminus A$. Using Corollary \ref{directionaltodiff}, choose a $\sigma$-porous set $P$ such that $x\notin P$ and existence of $X_{i}f(x)$ for $1\leq i\leq m$ together imply that $f$ is Pansu differentiable at $x$. Since porous sets have measure zero, the set $A\cup P$ has measure zero. The function $f$ is Pansu differentiable outside $A\cup P$, consequently Pansu differentiable almost everywhere.

Now assume $N>1$ and let $f=(f_1,\ldots, f_N)\colon \mathbb{G}\to \mathbb{R}^N$ be Lipschitz. Clearly $f_i\colon \bbG\to \bbR$ is Lipschitz for all $1\leq i\leq N$, therefore there exists a null set $B\subset \mathbb{G}$ so that every component $f_i$ is Pansu differentiable in $\mathbb{G}\setminus B$. Let $L\colon \mathbb{G}\to \mathbb{R}^N$ be the group linear map defined by $L=(L_1,\ldots, L_N)$ where $L_i$ is the Pansu derivative of $f_i$. For every $x\in \mathbb{G}\setminus B$ we have:
\begin{align*}
\lim_{h\to 0} \frac{|f(xh)-f(x)-L(h)|}{d(h)}&=\Big(\sum_{i=1}^N\lim_{h\to 0} \frac{(f_i(xh)-f_i(x)-L_i(h))^2}{d(h)^2}\Big)^{\frac{1}{2}}=0.
\end{align*}
Hence $f$ is Pansu differentiable almost everywhere.
\end{proof}

\end{document}